\documentclass[12pt]{amsart}
\usepackage{amssymb,mathrsfs}
\usepackage[matrix,arrow,curve]{xy}
\usepackage{setspace}
\newtheorem{theorem}[equation]{Theorem}
\newtheorem*{theorem*}{Theorem}
\newtheorem{lemma}[equation]{Lemma}
\newtheorem{corollary}[equation]{Corollary}
\newtheorem{conjecture}[equation]{Conjecture}
\newtheorem*{conjecture-bab}{Conjecture BAB}
\newtheorem{question}[equation]{Question}
\newtheorem{proposition}[equation]{Proposition}
\newtheorem{lemma-definition}[equation]{Lemma-Definition}

\theoremstyle{definition}

\newtheorem{definition}[equation]{Definition}
\newtheorem*{definition*}{Definition}

\theoremstyle{remark}
\newtheorem{remark}[equation]{Remark}

\makeatletter\@addtoreset{equation}{section}
\makeatletter\@addtoreset{section}{part}

\makeatother

\newcommand{\OOO}{\mathscr{O}}
\def \G {\mathcal{G}}
\def \X {\mathcal{X}}
\def \AA {\mathcal{A}}
\def \BB {\mathcal{B}}

\def \EE {\mathcal{E}}
\def \P {\mathbb{P}}

\def \Q {\mathbb{Q}}
\def \C {\mathbb{C}}
\def \Z {\mathbb{Z}}

\def \K {\mathbb{K}}

\newcommand{\Bir}{\operatorname{Bir}}
\newcommand{\Aut}{\operatorname{Aut}}
\newcommand{\Alb}{\operatorname{Alb}}
\newcommand{\Cl}{\operatorname{Cl}}
\newcommand{\Pic}{\operatorname{Pic}}
\newcommand{\NS}{\operatorname{NS}}
\newcommand{\GL}{\operatorname{GL}}
\newcommand{\PGL}{\operatorname{PGL}}
\newcommand{\ord}{\operatorname{ord}}
\newcommand{\W}{\operatorname{W}}
\newcommand{\red}{\operatorname{red}}
\newcommand{\tors}{\operatorname{tors}}
\newcommand{\nun}{\operatorname{nu}}
\newcommand{\alg}{\operatorname{alg}}
\newcommand{\ab}{\operatorname{ab}}
\newcommand{\RC}{\operatorname{rc}}

\def \ge {\geqslant}
\def \le {\leqslant}

\title{Jordan property for groups of birational selfmaps}

\author{Yuri Prokhorov}
\author{Constantin Shramov}

\address{Steklov Institute of Mathematics,
8 Gubkina street, Moscow 119991, Russia}

\address{Laboratory of Algebraic Geometry, GU-HSE, 7 Vavilova street,
Moscow 117312, \mbox{Russia}}

\email{prokhoro@gmail.com}

\email{costya.shramov@gmail.com}

\thanks{
Both authors were partially supported by the grants
RFBR-11-01-00336-a, N.Sh.-5139.2012.1,
and AG Laboratory SU-HSE, RF~government
grant ag.~11.G34.31.0023.
The first author was partially supported by
Simons-IUM fellowship.
The second author was partially supported by the grants
\mbox{MK-6612.2012.1}, RFBR-11-01-00185-a and RFBR-12-01-33024.
}

\begin{document}

\begin{abstract}
Assuming a particular case of Borisov--Alexeev--Borisov conjecture,
we prove that finite subgroups of
the automorphism group of a finitely generated 
field over~$\Q$ have bounded orders.
Further, we investigate which algebraic varieties have
groups of birational selfmaps satisfying
the Jordan property.
\end{abstract}

\maketitle


Unless explicitly stated otherwise, all varieties are assumed to be
algebraic, geometrically irreducible and defined 
over an \emph{arbitrary} field
$\Bbbk$ of characteristic zero.

\section{Introduction}
\label{section:intro}
This paper is motivated by two questions of J.-P.\,Serre
(\cite{Serre2009}, \cite{Edinburgh-2010})
concerning the finite subgroups of automorphism
groups of fields of characteristic zero.

Our starting point is the following.

\begin{question}[{J.-P.\,Serre \cite{Edinburgh-2010}}]
\label{question:Serre-Q}
Let $K$ be a finitely generated field over~$\Q$.
Is it true that there is a constant $B=B(K)$ such that
any finite subgroup~\mbox{$G\subset\Aut(K)$} has order~\mbox{$|G|\le B$}?
\end{question}

We will refer to the property mentioned in Question~\ref{question:Serre-Q}
as \emph{boundedness of finite subgroups}.

\begin{definition}[{cf.~\cite[Definition~2.9]{Popov2011}}]
\label{definition:BFS}
Let $\G$ be a family of groups. We say that
$\G$ \emph{has uniformly bounded finite subgroups}
if there exists a constant $B=B(\G)$ such that
for any~\mbox{$\Gamma\in\G$} and for any finite subgroup
$G\subset\Gamma$ one has $|G|\le B$.
We say that a group $\Gamma$
\emph{has bounded finite subgroups}
if the family $\{\Gamma\}$ has uniformly bounded finite subgroups.
\end{definition}

To answer Question~\ref{question:Serre-Q} we will translate it to
geometrical language. In some of our arguments
we will rely on a particular case of the well-known
Borisov--Alexeev--Borisov conjecture (see~\cite{Borisov-1996}).

\begin{conjecture-bab}
\label{conjecture:BAB}
Let $\bar{\Bbbk}$ be an algebraically closed field. For a given
positive integer $n$, Fano varieties
of dimension $n$ with terminal singularities
defined over~$\bar{\Bbbk}$
are bounded, i.\,e. are contained in a finite
number of algebraic families.
\end{conjecture-bab}

\begin{remark}
Note that if Conjecture~BAB
holds in dimension
$n$, then it holds in any dimension $m\le n$.
\end{remark}

The first main result
of our paper is as follows.

\begin{theorem}\label{theorem:Q}
Suppose that $\Bbbk$ is a finitely generated field over $\Q$.
Let $X$ be a variety of dimension~$n$.
Suppose that Conjecture~BAB
holds in dimension $n$.
Then the group~\mbox{$\Bir(X)$} of birational automorphisms of $X$
over $\Bbbk$ has bounded finite subgroups.
\end{theorem}

\begin{corollary}\label{corollary:Serre-OK}
The answer to Question~\ref{question:Serre-Q}
is positive modulo Conjecture~BAB
(cf.~Corollary~\ref{corollary:dim-3} below).
\end{corollary}

Besides boundedness of finite subgroups,
there is a somewhat analogous property of groups
that has recently attracted attention of algebraic geometers.

\begin{definition}[{cf.~\cite[Definition~2.1]{Popov2011}}]
\label{definition:uniformly-Jordan}
Let $\G$ be a family of groups. We say that
$\G$ is \emph{uniformly Jordan}
if there is a constant $J=J(\G)$ such that
for any group $\Gamma\in\G$
and any finite subgroup $G\subset\Gamma$ there exists
a normal abelian subgroup $A\subset G$ of index at most $J$.
We say that a group $\Gamma$
is \emph{Jordan}
if the family $\{\Gamma\}$ is uniformly Jordan.
\end{definition}

The classically known examples of Jordan groups include
$\GL_m(\mathbb C)$,
pointed out by C.\,Jordan (see e.\,g.~\cite[Theorem 36.13]{Curtis1962}),
and thus all linear algebraic groups
over an arbitrary field of characteristic zero.
J.-P.\,Serre proved that the group of
birational automorphisms of the projective plane
$\P^2$ over a field of characteristic zero is Jordan
(see \cite[Theorem~5.3]{Serre2009}),
and asked if the same holds for groups of
birational automorphisms of projective spaces $\P^n$ for $n\ge 3$
(see \cite[6.1]{Serre2009}).
Recently the authors were able to establish the following
theorem that deals with the case of rationally
connected varieties (see e.\,g.~\mbox{\cite[IV.3.2]{Kollar-1996-RC}})
of arbitrary dimension, and in particular answers the latter question.

\begin{theorem}[{\cite[Theorem~1.8]{Prokhorov-Shramov}}]
\label{theorem:RC-Jordan}
Let $\G_{\RC}(n)$ be the family of groups $\Bir(X)$, where~$X$
varies in the set
of rationally connected varieties of dimension $n$.
Assume that Conjecture~BAB
holds in dimension $n$.
Then the family $\G_{\RC}(n)$ is uniformly Jordan.
\end{theorem}

Yu.\,G.\,Zarhin found an example of a surface $X$ such that
the group~\mbox{$\Bir(X)$} is not Jordan (see~\cite{Zarhin10}), and
V.\,L.\,Popov classified all surfaces $X$ such that~\mbox{$\Bir(X)$} is Jordan
(see~\cite[Theorem~2.32]{Popov2011}).
The next natural step may be to wonder if it is possible to do
something similar in higher dimensions.
The second main result of our current paper is the following
theorem that completely solves
the question for
non-uniruled varieties (see~\cite[\S IV.1.1]{Kollar-1996-RC})
and partially describes the general case.
Recall that \emph{irregularity} of 
a variety~$X$ is defined as~\mbox{$q(X)=\dim H^1(X',\OOO_{X'})$}, where~$X'$ 
is any smooth projective variety birational to~$X$.

\begin{theorem}\label{theorem:main}
Let $X$ be a variety of dimension $n$.
Then the following assertions hold.
\begin{itemize}
\item[(i)]
The group $\Bir(X)$ has bounded
finite subgroups provided that $X$ is
non-uniruled and has irregularity
$q(X)=0$.
\item[(ii)]
The group $\Bir(X)$ is Jordan
provided that $X$ is non-uniruled.
\item[(iii)]
Suppose that Conjecture~BAB
holds in dimension $n$. Then the group $\Bir(X)$ is Jordan
provided that $X$ has irregularity
$q(X)=0$.
\end{itemize}
\end{theorem}

Note that Conjecture~BAB
is proved in dimension $n\le 3$ (see~\cite{KMMT-2000}).
Therefore, one has the following.

\begin{corollary}[{cf.~\cite[Corollary~1.9]{Prokhorov-Shramov}}]
\label{corollary:dim-3}
Theorems~\ref{theorem:Q} and~\ref{theorem:main}
(as well as Theorem~\ref{theorem:RC-Jordan})
hold in dimension $n\le 3$
without any additional assumptions.
\end{corollary}

\begin{remark}\label{remark:Zarhin}
Yu.\,G.\,Zarhin showed in~\cite{Zarhin10} that the group of birational selfmaps
of a variety that is isomorphic to a product
of an abelian variety and a projective line over an algebraically
closed field of characteristic zero
violates the Jordan property.
This shows that one cannot remove the conditions of non-uniruledness and
vanishing irregularity from Theorem~\ref{theorem:main}(ii),(iii).
Moreover, by~\cite[Theorem 2.32]{Popov2011}
the only surface $X$ over an algebraically
closed field of characteristic zero such that
$\Bir(X)$ fails to have the Jordan property is the product $E\times\P^1$,
where $E$ is an elliptic curve. Thus (to a certain extent) we may consider
Theorem~\ref{theorem:main}(ii),(iii) to be
a generalization of~\cite[Theorem 2.32]{Popov2011}.
\end{remark}

Besides the results listed above, in~\S\ref{section-Solvably-Jordan-groups} 
we discuss
a ``solvable'' analog of the Jordan property.
The main result there is
Proposition~\ref{proposition:solvably-Jordan}
which answers
Question~\ref{question:solvable} asked by D.\,Allcock.

\begin{remark}
Note that the group $\Bir(X)$ of birational automorphisms
of a variety $X$ over $\Bbbk$ has a structure of a $\Bbbk$-scheme,
although in general it is not a group scheme, and it is not a birational 
invariant of $X$ (see~\cite[\S1]{Hanamura1988}). 
However, both of these properties hold in an important particular
case when $X$ is a minimal model (see~\cite[Theorem~3.3(1)]{Hanamura1987}
and~\cite[Theorem~3.7(2)]{Hanamura1987}).
If $X$ is non-uniruled, then the structure of $\Bir(X)$ 
is much better understood than in the general case.
In particular, it is known that if $X$ is non-uniruled, then
the dimension of $\Bir(X)$ is at most $q(X)$ 
(see~\cite[Theorem~2.1(i)]{Hanamura1988}).
Moreover, there is an interesting conjecture that may be 
relevant to Theorem~\ref{theorem:main}(i): if~$X$ is non-uniruled
and satisfies some additional assumptions, then 
the ``discrete part'' $\Bir(X)_{\red}/\Bir(X)^0$ is finitely generated
(see~\cite[\S7.4]{Hanamura1988}).
Since the general structure of $\Bir(X)$ is not a subject of this paper,
we refer the reader to~\cite{Hanamura1987}, \cite{Hanamura1988} 
and references therein for further information.
\end{remark}

\smallskip
The plan of the paper is as follows. In~\S\ref{section:basics}
we discuss the basic (group-theoretical) properties of Jordan groups
and groups with bounded finite subgroups, and also collect some important
examples of groups of each of these two classes.
In~\S\ref{section:groups-divisors} we recall some well-known 
auxiliary geometrical facts.
In~\S\ref{section:quasi-minimal} we introduce (following a suggestion of
Caucher Birkar) quasi-minimal models
that are analogs of minimal models such that
one does not need the full strength of
the Minimal Model Program to prove their existence.
In~\S\ref{section:minimal-models} we discuss the action of finite
groups on quasi-minimal models. In~\S\ref{section:proofs} we prove
Proposition~\ref{proposition:technical}, which is our main auxiliary
result describing the general structure of finite groups
of birational automorphisms, and use it to derive
Theorem~\ref{theorem:main}. In~\S\ref{section:Q}
we prove Theorem~\ref{theorem:Q} and derive
Corollary~\ref{corollary:Serre-OK}. In~\S\ref{section-Solvably-Jordan-groups}
we discuss solvably Jordan groups. Finally, in~\S\ref{section:discussion}
we discuss some open questions related to the subject of this paper.

\smallskip
\textbf{Acknowledgements.}
We would like to thank J.-P.\,Serre who asked us questions considered here.
We are grateful to P.\,Cascini, F.\,Catanese, I.\,A.\,Cheltsov,
S.\,O.\,Gorchinskiy, A.\,N.\,Parshin,
V.\,L.\,Popov and Yu.\,G.\,Zarhin for useful discussions.
Special thanks go to
C.\,Birkar who explained us the concepts of~\S\ref{section:quasi-minimal}, to
A.\,A.\,Klyachko who explained us the proof of
Lemma~\ref{lemma:Jordan-by-BFS-with-BR}, to De-Qi Zhang who pointed out that 
the first version of our proof of 
Lemma~\ref{lemma:solvably-Jordan-by-solvably-Jordan}
was insufficient, and to L.\,Pyber who explained us how to complete it.
We also thank the referee who provided many valuable comments on the 
paper. This work was mostly completed while the authors were visiting 
National Center for 
Theoretical Sciences~\mbox{(NCTS/TPE)} and National Taiwan University.
We are grateful to these institutions and personally to
Jungkai Chen for invitation and hospitality.

\section{Basic properties}
\label{section:basics}

\begin{remark}\label{remark:BFS-implies-Jordan}
If a family $\G$ has uniformly bounded finite subgroups,
then it is uniformly Jordan.
\end{remark}

\begin{lemma}
\label{lemma:BFS-by-BFS}
Let $\G_1$ and $\G_2$ be families of groups with uniformly bounded
finite subgroups.
Let $\G$ be a family of groups $G$ such that there exists
an exact sequence
$$1\longrightarrow G_1\longrightarrow G\longrightarrow G_2\longrightarrow 1,$$
where $G_1\in\G_1$ and $G_2\in\G_2$.
Then $\G$ has uniformly bounded finite subgroups.
\end{lemma}
\begin{proof}
Straightforward.
\end{proof}

\begin{lemma}
\label{lemma:Jordan-by-BFS}
Let $\G_1$ and $\G_2$ be families of groups
such that $\G_1$ is uniformly Jordan and~$\G_2$ has uniformly bounded
finite subgroups.
Let $\G$ be a family of groups $G$ such that there exists
an exact sequence
$$1\longrightarrow G_1\longrightarrow G\longrightarrow G_2\longrightarrow 1,$$
where $G_1\in\G_1$ and $G_2\in\G_2$.
Then $\G$ is uniformly Jordan.
\end{lemma}
\begin{proof}
See \cite[Lemma~2.11]{Popov2011}.
\end{proof}

\begin{remark}[{cf. \cite[Remark~2.12]{Popov2011}}]
\label{remark:oops}
If $\G_1$ and $\G_2$ are families of groups
such that $\G_1$ has uniformly bounded
finite subgroups and $\G_2$ is uniformly Jordan, then
a family $\G$ of groups $G$ such that there exists
an exact sequence
$$1\longrightarrow G_1\longrightarrow G\longrightarrow G_2\longrightarrow 1$$
with $G_1\in\G_1$ and $G_2\in\G_2$
may fail to be uniformly Jordan.
For example, this is the case
if $\G_1$ consists of a single group $\Z/p\Z$, where $p$ is a prime,
and $\G_2$ consists of groups of the form $\big(\Z/p\Z\big)^{2r}$
for various $r$.
\end{remark}

For applications in~\S\ref{section:proofs}
we would like to know some additional condition that would
guarantee that the extensions considered in Remark~\ref{remark:oops}
form a uniformly Jordan family (cf.~\cite[Corollary~2.13]{Popov2011}).
One of such conditions relies on the following auxiliary
definition.

\begin{definition}
\label{definition:bounded-rank}
Let $\G$ be a family of groups. We say that
$\G$ \emph{has finite subgroups of uniformly bounded rank}
if there exists a constant $R=R(\G)$ such that
for any $\Gamma\in\G$ each finite abelian subgroup
$A\subset\Gamma$ is generated by at most $R$ elements.
We say that a group~$\Gamma$
\emph{has finite subgroups of bounded rank}
if the family $\{\Gamma\}$ has
finite subgroups of uniformly bounded rank.
\end{definition}

\begin{remark}\label{remark:stupid}
If a family $\G$ has uniformly bounded finite subgroups,
then it has
finite subgroups of uniformly bounded rank.
\end{remark}

\begin{lemma}
\label{lemma:BR-by-BR}
Let $\G_1$ and $\G_2$ be families of groups
with finite subgroups of uniformly bounded rank.
Let $\G$ be a family of groups $G$ such that there exists
an exact sequence
$$1\longrightarrow G_1\longrightarrow G\longrightarrow G_2\longrightarrow 1,$$
where $G_1\in\G_1$ and $G_2\in\G_2$.
Then $\G$
has finite subgroups of uniformly bounded rank.
\end{lemma}
\begin{proof}
Straightforward.
\end{proof}

An advantage of Definition~\ref{definition:bounded-rank} is
the following property explained to us by Anton Klyachko.

\begin{lemma}
\label{lemma:Jordan-by-BFS-with-BR}
Let $\G_1$ and $\G_2$ be families of groups
such that $\G_1$ has uniformly bounded
finite subgroups and $\G_2$ is uniformly Jordan and has
finite subgroups of uniformly bounded rank.
Let $\G$ be a family of groups $G$ such that there exists
an exact sequence
$$1\longrightarrow G_1\longrightarrow G\longrightarrow G_2\longrightarrow 1,$$
where $G_1\in\G_1$ and $G_2\in\G_2$.
Then $\G$ is uniformly Jordan.
\end{lemma}
\begin{proof}
It is enough to prove that
if $F$ is a finite group and $A$ is a finite abelian group
generated by $r$ elements, then for any extension
$$1\longrightarrow F\longrightarrow G\longrightarrow A\longrightarrow 1$$
one can bound the index $[G:Z]$ of the center $Z$ of the group $G$
in terms of $|F|$ and $r$.

Let $K\subset G$ be the commutator subgroup.
Since $A$ is abelian, one has
$$|K|\le \frac{|G|}{|A|}=|F|.$$
For $x\in G$ let $K(x)$ be the set of the commutators
of the form $gxg^{-1}x^{-1}$ for various $g\in G$,
and $Z(x)\subset G$ be the centralizer of $x$.
It is easy to see that
$$[G:Z(x)]=|K(x)|\le |K|.$$
Now if
$\{x_1,\ldots, x_N\}\subset G$ is a subset that generates $G$,
then $Z=Z(x_1)\cap\ldots\cap Z(x_N)$, so that
$$[G:Z]\le [G:Z(x_1)]\cdot\ldots\cdot [G:Z(x_N)]\le
|K(x_1)|\cdot\ldots\cdot |K(x_N)|\le |K|^N\le |F|^N.$$
It remains to notice that one can choose a set of
$N\le r|F|$ generators for the group $G$.
\end{proof}

\begin{remark}
Let $\G$ be a family of
groups, and let $\tilde{\G}$ be the family that consists
of all finite subgroups of all groups in $\G$. Then
the family $\G$ has bounded finite subgroups
(resp., is uniformly Jordan, has finite subgroups
of uniformly bounded rank) if and only if
the family $\tilde{\G}$ has bounded finite subgroups
(resp., is uniformly Jordan, has finite subgroups
of uniformly bounded rank). We will sometimes use this trivial 
observation without any further comments 
while applying Lemmas~\ref{lemma:BFS-by-BFS},
\ref{lemma:Jordan-by-BFS}, \ref{lemma:BR-by-BR} 
and~\ref{lemma:Jordan-by-BFS-with-BR}.
\end{remark}


Now we will discuss some important
examples of groups with bounded finite subgroups
and of Jordan groups.
We will use the following notation.

\begin{definition}\label{definition:Aut-g}
Let $A$ be an abelian variety over $\Bbbk$.
By $\Aut_{g}(A)$ we denote the group of
automorphisms of $A$ as a $\Bbbk$-variety
(i.\,e. the group of automorphisms of the variety $A$ that
may not respect the group structure on $A$).
\end{definition}

\begin{remark}\label{remark:Aut-g}
One has
$$\Aut_{g}(A)\simeq A(\Bbbk)\rtimes\Gamma,$$
where $A(\Bbbk)$ denotes the group
of $\Bbbk$-points of the abelian variety $A$ and $\Gamma$ is a subgroup
of~$\GL_{2\dim(A)}(\Z)$.
\end{remark}

The following is a well-known theorem of H.\,Minkowski
(see e.\,g.~\cite[Theorem~5]{Serre2007}
and~\cite[\S4.3]{Serre2007}).

\begin{theorem}
\label{theorem:GL-number-field}
Suppose that $\Bbbk$ is a finitely generated field over $\Q$,
and~$m$ is a positive integer.
Then the group~\mbox{$\GL_m(\Bbbk)$} has
bounded finite subgroups.
\end{theorem}

\begin{corollary}\label{corollary:GLZ}
The group $\GL_m(\Z)$ has
bounded finite subgroups.
\end{corollary}

\begin{corollary}
\label{corollary:lattice}
Let $N$ be a finitely generated abelian group.
Then the group
$\Aut(N)$ has bounded finite subgroups.
\end{corollary}
\begin{proof}
One has an exact sequence
$$0\longrightarrow T\longrightarrow N\longrightarrow N/T\longrightarrow 0,$$
where $T$ is the torsion subgroup of $N$ and $N/T$ is a free abelian group
of finite rank $r$.
Therefore, one has an exact sequence
$$0\longrightarrow T^r\longrightarrow \Aut(N)\longrightarrow \Aut(T)\times\Aut(N/T)\longrightarrow 1.$$
The group $\Aut(N/T)\simeq\GL_r(\Z)$ has bounded finite subgroups
by Corollary~\ref{corollary:GLZ},
while the groups $T^r$ and $\Aut(T)$ are finite.
Now the assertion follows by Lemma~\ref{lemma:BFS-by-BFS}.
\end{proof}

\begin{corollary}
\label{corollary:abelian-variety-Jordan}
Let $\AA_d$ be the family of groups $\Aut_{g}(A)$, where $A$ varies
in the set
of abelian varieties of dimension $d$ over a field $\Bbbk$,
while $\Bbbk$ varies in the set of all fields of characteristic zero.
Then $\AA_d$ is uniformly Jordan and has finite subgroups of uniformly 
bounded rank.
\end{corollary}
\begin{proof}
To prove that $\AA_d$ is uniformly Jordan, apply Remark~\ref{remark:Aut-g},
Corollary~\ref{corollary:GLZ}
and Lemma~\ref{lemma:Jordan-by-BFS}.

Let us prove that $\AA_d$ has finite subgroups of uniformly bounded rank.
Let $A$ be an abelian variety of dimension $d$. Then for any positive integer
$n$ the $n$-torsion subgroup of the group $A\big(\bar{\Bbbk}\big)$, where
$\bar{\Bbbk}$ is the algebraic closure of $\Bbbk$,
is isomorphic to $\big(\Z/n\Z\big)^{2d}$, which implies that
any finite subgroup of
$A(\Bbbk)$ is generated by at most $2d$ elements.
Furthermore, any finite abelian
subgroup of $\GL_{m}\big(\bar{\Bbbk}\big)$
is diagonalizable, so that
any finite abelian subgroup $H\subset\GL_{2d}(\Z)$
is also generated by at most $2d$
elements (alternatively, one can use Corollary~\ref{corollary:GLZ}
to deduce that $\GL_{2d}(\Z)$ has finite subgroups of bounded rank). 
Now the assertion follows by
Remark~\ref{remark:Aut-g} and Lemma~\ref{lemma:BR-by-BR}.
\end{proof}

\begin{corollary}
\label{corollary:abelian-variety-BFS}
Suppose that $\Bbbk$ is a finitely generated field over $\Q$.
Let $A$ be an abelian variety over $\Bbbk$.
Then the group $\Aut_{g}(A)$ has bounded finite subgroups.
\end{corollary}
\begin{proof}
Recall that the group $A(\Bbbk)$ of $\Bbbk$-points of
$A$ is a finitely generated abelian group by the Mordell--Weil theorem
(see~\cite[Chapter~6, Theorem~1]{Lang-1983-book}).
Thus $A(\Bbbk)$ has bounded finite subgroups,
so that the assertion follows by Remark~\ref{remark:Aut-g},
Corollary~\ref{corollary:GLZ}
and Lemma~\ref{lemma:BFS-by-BFS}.
\end{proof}

When $\Bbbk$ is a number field, it is expected that
a stronger version of Corollary~\ref{corollary:abelian-variety-BFS}
holds. The starting point here is the following.

\begin{conjecture}
\label{conjecture:torsion}
Suppose that $\Bbbk$ is a finitely generated field over $\Q$, and
$d$ is a positive integer.
Then there is a constant $M=M(d)$ such that for any abelian variety
$A$ of dimension $d$ the order of the torsion subgroup
$A(\Bbbk)_{\tors}$ in $A(\Bbbk)$ is less than $M$.
\end{conjecture}

Note that Conjecture~\ref{conjecture:torsion} is not universally
recognized as credible (cf.~\cite[Question~2.1]{Poonen12};
see also~\cite[Boundedness Conjecture]{MortonSilverman94}
and~\cite[Corollary~2.4]{Fakhruddin03}).
Conjecture~\ref{conjecture:torsion}
is proved only for dimension $d=1$, i.\,e. for elliptic curves.
The case of a number field was established in~\cite{Merel},
and the case of an arbitrary field
$\Bbbk$ finitely generated over $\Q$ is derived from it
in a standard way (see e.\,g. the remark made after Question~2.1
in~\cite{Poonen12}).

Modulo Conjecture~\ref{conjecture:torsion} one has the following
refined version of Corollary~\ref{corollary:abelian-variety-BFS}.

\begin{corollary}
\label{corollary:abelian-variety-BFS-torsion}
Suppose that $\Bbbk$ is a finitely generated field over $\Q$.
Let $\AA_d(\Bbbk)$ be the family 
of groups~\mbox{$\Aut_{g}(A)$}, where $A$ varies
in the set
of abelian varieties of dimension~$d$ over~$\Bbbk$.
Suppose that Conjecture~\ref{conjecture:torsion} holds in dimension $d$
over $\Bbbk$.
Then $\AA_d(\Bbbk)$ has uniformly bounded finite subgroups.
\end{corollary}

\section{Divisor class groups}
\label{section:groups-divisors}

The following notion is well-known and widely used (cf.~\cite[\S4]{Prokhorov-Shramov}).

\begin{lemma-definition}\label{regularization}
Let $X$ be a variety
and \mbox{$G\subset\Bir(X)$} be a finite group.
There exists a normal projective variety $\tilde X$
with a biregular action of $G$ and a $G$-equivariant birational
map $\tilde X \dashrightarrow X$.
The variety $\tilde X$ is called
a \emph{regularization}
of $G$. Moreover, $\tilde X$  can be taken smooth
and then $\tilde X$ is called
a \emph{smooth regularization}
of $G$.
\end{lemma-definition}
\begin{proof}
By shrinking $X$ we may assume that $X$ is affine and
$G$ acts on $X$ biregularly.
Then the quotient $V=X/G$ is also affine,
so there exists a projective completion $\tilde V\supset V$.
Let~$\tilde{X}$ be the normalization of $\tilde V$ in the function field
$\Bbbk(X)$. Then $\tilde X$ is a projective variety \mbox{$G$-birational}
to~$X$, and $\tilde{X}$ admits a biregular action of $G$.
Taking a $G$-equivariant resolution
of singularities (see~\cite{Bierstone2008}),
one can assume that $\tilde{X}$ is smooth.
\end{proof}

Let $X$ be a normal projective variety.
By $\Cl(X)$ we denote
the group of Weil divisors on~$X$ modulo linear equivalence, and by
$\Cl^0(X)$
its subgroup consisting of divisors that are algebraically equivalent to $0$.
Let $f\colon\tilde X\to X$
be a resolution of singularities. It induces a natural map
$$f_*\colon\Cl^0(\tilde X) \longrightarrow \Cl^0(X).$$
The group
$\Cl(X)$ is canonically
isomorphic to the quotient of $\Cl(\tilde X)$
by the subgroup~\mbox{$\EE\subset\Cl(\tilde{X})$}
generated by $f$-exceptional divisors.
Since prime exceptional divisors are linearly independent modulo
numerical equivalence (see e.g. \cite[Lemma 2.19]{Utah}),
we have
$$\Cl^0(X)\cap\EE=0.$$
Hence $f_*\colon\Cl^0(\tilde X) \to \Cl^0(X)$ is an isomorphism.
In particular, $\operatorname{Cl}^0(X)$
is a birational invariant in the
category of projective varieties.
Moreover, $f_*$ induces a structure of an abelian variety on $\Cl^0(X)$.
The group
$$\NS^{\W}(X)= \operatorname{Cl}(X) /\operatorname{Cl}^0(X)$$
is a homomorphic image of the group
$$\NS^{\W}(\tilde{X})=\Cl(\tilde{X})/\Cl^0(\tilde{X})
\simeq\Pic(\tilde X)/\Pic^0(\tilde X),$$
which is finitely generated
by the Neron--Severi theorem.
Therefore, $\NS^{\W}(X)$ is also finitely generated.
Slightly abusing the standard terminology,
we will refer to the group~$\NS^{\W}(X)$ as \emph{the Neron--Severi
group} of $X$.

\begin{remark}
Let $G\subset \Bir(X)$ be a finite subgroup. By
Lemma-Definition~\ref{regularization}
we can choose a regularization $\tilde{X}$ of $G$, so that
$G\subset \Aut(\tilde{X})$. Since $\Pic^0$
is a functor, the group $G$ naturally acts on
$\Cl^0(X)\simeq \Pic^0(\tilde{X})$, and this action
does not depend on our choice of the resolution~$f$.
\end{remark}

\section{Quasi-minimal models}
\label{section:quasi-minimal}

Starting from this point we will use standard terminology
and conventions of the Minimal Model Program
(see e.\,g. \cite{KMM} or \cite{Matsuki2002}).
We note that there exist natural generalizations of the Minimal Model
Program to 
the cases of varieties over non-closed field 
and varieties with group action.
Since these notions are quite standard 
(see e.\,g.~\mbox{\cite[\S2.2]{Kollar-Mori-1988}}), 
we will refer to the recent results of~\cite{BCHM}
concerning the Minimal 
Model Program without further comments on these different
setups.

In this section we introduce the notion of quasi-minimal models,
following the idea of Caucher Birkar. This is a weaker analog
of a usual notion of minimal models which has an advantage that
to prove its existence we do not need the full strength of
the Minimal Model Program.

\begin{definition}\label{definition:quasi-minimal}
An effective $\Q$-divisor $M$ on a variety $X$
is said to be \textit{$\Q$-movable} if for some $n>0$
the divisor $nM$ is integral and generates a linear system
without fixed components.
\end{definition}

\begin{definition}\label{definition:qm}
Let $X$ be a projective variety with terminal singularities.
We say that~$X$ is a \emph{quasi-minimal model}
if there exists a sequence of $\Q$-movable $\Q$-Cartier 
$\Q$-divisors $M_j$
whose limit in the Neron--Severi space
$\NS_{\Q}^{\W}(X)=\NS^{\W}(X)\otimes\Q$ is $K_X$.
\end{definition}

\begin{remark}\label{remark:qm-non-uniruled}
Any minimal model is a quasi-minimal model by Kleiman ampleness criterion.
By~\cite[Theorem~1]{MiyaokaMori} any quasi-minimal model is non-uniruled.
\end{remark}

Now we will show that the current state of art in the Minimal
Model Program allows to prove the existence of quasi-minimal models.
Recall that a (normal) 
variety $X$ acted on by a finite group $\Gamma$
has \emph{$\Gamma\Q$-factorial} singularities, if and only if any
$\Gamma$-invariant Weil divisor on $X$ is $\Q$-Cartier.

\begin{lemma}\label{lemma-quasi-minimal-model}
Let $X$ be a projective non-uniruled variety with terminal singularities,
and~\mbox{$\Gamma\subset\Aut(X)$} be a finite subgroup.
Assume that $X$ has $\Gamma\Q$-factorial singularities.
Then there exists a
quasi-minimal model $X'$ birational to $X$
such that $\Gamma\subset\Aut(X')$.
\end{lemma}
\begin{proof}
Run a $\Gamma$-equivariant Minimal Model Program
on~$X$.
Since $X$ is non-uniruled, we will never arrive to a non-birational
contraction by~\mbox{\cite[Corollary~5-1-4]{KMM}}.
Thus, if this $\Gamma$-equivariant Minimal Model Program
terminates, then it gives a minimal model (that is in particular,
a quasi-minimal model) $X'$ birational to $X$ such that 
$\Gamma\subset\Aut(X')$.

We use induction in the Picard number $\rho(X)=\dim\NS^{\W}_{\Q}(X)$.
If $\rho(X)=1$, then $X$ is a minimal model
(and in particular, a quasi-minimal model) itself. If some step 
of the $\Gamma$-Minimal Model Program is a divisorial contraction,
then the Picard number drops at least by one at this step,
and we proceed by induction.
The only disaster that may happen is that 
the $\Gamma$-equivariant Minimal Model Program ran on $X$ 
does not terminate, and each of its steps
is a $\Gamma$-flip.
We claim that in this case $X$ is a quasi-minimal model.

Take a very ample $\Gamma$-invariant divisor $A$ on $X$ and a
sequence of positive numbers~$t_j$ approaching  $0$.
According to \cite{BCHM} 
(or rather the $\Gamma$-equivariant versions of the 
corresponding theorems) we can run
a $\Gamma$-equivariant $(K_X+t_jA)$-Minimal Model Program
on~$X$ with scaling of $A$
to obtain a $\Gamma$-equivariant birational map
$$
\psi_j\colon (X, t_jA) \dashrightarrow (X_j, t_jA_j).
$$
Since $X$ is not uniruled, $(X_j, t_jA_j)$ is a log minimal model.
By the construction of the Minimal Model Program with scaling (see~\cite{BCHM})
all extremal rays of $\psi_j$ are~\mbox{$A$-positive}.
Hence, they are $K$-negative, and so $\psi_j$ is a composition of
$\Gamma$-flips.
Since the \mbox{$\Q$-divisor}~\mbox{$K_{X_j}+t_jA_j$}
is nef, it is a limit of $\Q$-movable 
(and even ample) $\Q$-Cartier $\Q$-divisors by 
Kleiman ampleness criterion.
On the other hand, the varieties $X_j$ are isomorphic in codimension~$1$,
so that the Neron--Severi spaces $\NS^{\W}_{\Q}(X_j)$
are naturally identified with~\mbox{$\NS^{\W}_{\Q}(X)$} 
(cf.~\S\ref{section:minimal-models} below),
and the divisors $K_{X_j}\in\NS^{\W}_{\Q}(X_j)$
correspond to~\mbox{$K_X\in\NS^{\W}_{\Q}(X)$}. 
Therefore, the divisor~$K_{X}$
is also a limit of $\Q$-movable
$\Q$-Cartier $\Q$-divisors, i.\,e.~$X$ is a quasi-minimal model.
\end{proof}

\begin{remark}\label{remark:trivial-group-qm}
In this paper we will use Lemma~\ref{lemma-quasi-minimal-model}
only for a trivial group $\Gamma$ (nevertheless, this will 
allow us to make conclusions about certain non-trivial groups
acting on $X$). Still we prefer to give a more general 
form of the lemma since we believe that it may have other applications.
\end{remark}

Below we will establish an important property of quasi-minimal
models that they share with minimal models.

\begin{proposition}[cf. {\cite[Proof of Theorem~4.1, Step 3]{Birkar2012}}]
Let $X$ be a quasi-minimal model and let $\chi\colon X\dashrightarrow X'$
be any
birational map, where $X'$  has only terminal
singularities. Then $\chi$ does not contract any divisors.
\end{proposition}
\begin{proof}
Assume that $\chi$ contracts a (prime) divisor $D$.
Consider a common resolution
\[
\xymatrix{
&Z\ar[dr]^{g}\ar[dl]_{f}&
\\
X\ar@{-->}[rr]^{\chi}&&X'
}
\]
and let $D_Z\subset Z$ be the proper transform of $D$.
Clearly, $D_Z$ is $g$-exceptional.

Take a sequence of $\Q$-movable $\Q$-Cartier $\Q$-divisors $M_j$
whose limit in the Neron--Severi space
$\NS^{\W}_{\Q}(X)$ is $K_X$,
write
$$
K_Z\equiv f^*K_X+E,
$$
and put
$N_j=f^*M_j+E$. Since $X$ has terminal singularities,
the $\Q$-divisor $E$ is effective, and thus the $\Q$-divisor
$N_j$ is also effective. Moreover, $K_Z$ is the 
limit of the $\Q$-divisors~$N_j$
in the Neron--Severi space $\NS^{\W}_{\Q}(Z)$,
and we may assume that
$D_Z$ is not a component of~$N_j$.

Since $g$ is birational, by~\cite[Corollary~1.4.2]{BCHM}
we can run \mbox{$K$-Minimal Model}
Program over $X'$:
$$
Z=Z_1 \dashrightarrow \ldots \dashrightarrow Z_i\overset{p_i}{\dashrightarrow} Z_{i+1}
\dashrightarrow\ldots \dashrightarrow Z_n \longrightarrow X'.
$$
Since $X'$ has only terminal
singularities, the map $Z_n\to X'$ is a small $K$-trivial contraction.
We may assume that the proper transform of $D$ is contracted by $p_i$.
Thus $p_i$ is a morphism whose exceptional locus $D^{(i)}\subset Z_i$
is the proper transform of $D$.

Let $N_j^{(i)}$ be the proper transform of $N_j$ on $Z_i$.
Then $D^{(i)}$ is not a component of $N_j^{(i)}$. Moreover, $K_{Z_i}$
is a limit of $N_j^{(i)}$
in the Neron--Severi space $\NS^{\W}_{\Q}(Z_j)$.
The divisor $K_{Z_i}$ is strictly 
negative on the curves in fibers of $p_i$,
so that $N_j^{(i)}$ is also negative on them for~\mbox{$j\gg 0$}.
Note that one can choose 
an algebraic family of such curves covering $D^{(i)}$.
Thus~$D^{(i)}$ is a component of $N_j^{(i)}$
for $j\gg 0$, which is a contradiction.
\end{proof}

\begin{corollary}[{cf. \cite[Lemma 3.4]{Hanamura1987}}]
\label{corollary:2-qm}
Let $X$ and $Y$ be two quasi-minimal models. 
Then every birational map $\chi\colon X\dasharrow Y$
is an isomorphism in codimension one.
\end{corollary}

\section{Groups acting on quasi-minimal models}
\label{section:minimal-models}

Let $X$ be a quasi-minimal model (see Definition~\ref{definition:qm}), and
$G\subset\Bir(X)$ be a finite group.
By Corollary~\ref{corollary:2-qm},
any element $g\in G$ maps $X$ to itself isomorphically in codimension~$1$.
Thus $G$ acts on $\operatorname{Cl}(X)$ and on $\operatorname{Cl}^0(X)$.
Clearly, this induces an action of $G$ on
$\NS^{\W}(X)$, i.\,e. a homomorphism
$$\theta_{NS}\colon G\longrightarrow \Aut\big(\NS^{\W}(X)\big).$$
Moreover, the kernel
$$\operatorname{Ker}(\theta_{NS})\subset G$$
acts on any
algebraic equivalence class $\Cl_L(X)\subset\Cl(X)$
preserving the structure of an algebraic variety
on $\Cl_L(X)$.

\begin{remark}
In the above notation,
assume also that the field $\Bbbk$
is algebraically closed and $X$ is a minimal model. 
Then, according to~\cite[Theorem~3.3(1)]{Hanamura1987},
the group~\mbox{$\Bir(X)$} has
a natural structure of a group scheme.
Using this one can define an action of the whole group~\mbox{$\Bir(X)$}
on~\mbox{$\Cl(X)$} and~\mbox{$\Cl^0(X)$}. Since we are interested only in
finite group actions, we do not need these constructions,
and take an advantage of a more elementary approach that also does not
need additional assumptions on $\Bbbk$.
\end{remark}

\begin{lemma}\label{lemma-minimal-model}
Let $X$ be a quasi-minimal model.
Let $L$ be an ample $\Q$-Cartier divisor on~$X$. Then
the group $\Bir(X, L)$ of birational automorphisms of $X$
that preserve the class~\mbox{$[L]\in \Pic(X)$} is finite.
\end{lemma}
\begin{proof}
We may assume that $L$ is a very ample Cartier divisor.
Suppose that some element~\mbox{$\varphi\in\Bir(X)$} preserves the class
$[L]\in \Pic(X)\subset \operatorname{Cl}(X)$.
Since
$$X\simeq\operatorname{Proj}\bigoplus_{n\ge 0} H^0(X, nL),$$
the map $\varphi$ is in fact a biregular automorphism of $X$.
Therefore, $\Bir(X, L)$ is a subgroup of the group of
linear transformations of the projective space
$$\mathbb P\big(H^0(X,L)^\vee\big)\simeq\P^N,$$
so that
$$\Bir(X, L)\subset\PGL_{N+1}(\Bbbk).$$
If $\Bir(X,L)$ is not finite, then it contains a one-parameter
subgroup $G$, where $G\simeq \mathbb G_{\mathrm m}$
or~\mbox{$G\simeq\mathbb{G}_{\mathrm a}$}.
In particular, for a general point $P\in X$ the orbit $G\cdot P$
must be a geometrically
rational curve, so that $X$ is uniruled.
This contradicts the fact that a quasi-minimal model is non-uniruled
(see Remark~\ref{remark:qm-non-uniruled}).
\end{proof}

\section{Proof of Theorem~\ref{theorem:main}}
\label{section:proofs}

In this section we prove
Proposition~\ref{proposition:technical}, which is our main auxiliary
result describing the general structure of finite groups
of birational automorphisms, and use it to derive
Theorem~\ref{theorem:main}.

Recall that to any variety $X$
one can associate \emph{the maximal rationally connected fibration}
$$\phi_{\RC}\colon X\dasharrow X_{\nun},$$
which is a canonically defined rational map with
rationally connected fibers and non-uniruled base $X_{\nun}$
(see~\cite[\S IV.5]{Kollar-1996-RC},
\cite[Corollary 1.4]{Graber-Harris-Starr-2003}).

\begin{definition}\label{definition:polarization}
Let $X$ be a variety.
By \emph{a birational polarization on the base of the maximal
rationally connected fibration of $X$} we mean an ample divisor on
one of the quasi-minimal models (see Definition~\ref{definition:qm})
of the base $X_{\nun}$ of
the maximal rationally connected fibration~$\phi_{\RC}$.
\end{definition}

\begin{proposition}\label{proposition:technical}
Let $X$ be a variety of dimension $n$,
and let $G\subset\Bir(X)$ be
a finite subgroup.
Choose some birational polarization $L$ on the base of the maximal
rationally connected fibration of~$X$.
Then there exist exact sequences
\begin{equation}\label{seq1}
1\longrightarrow G_{\RC}\longrightarrow G\longrightarrow G_{\nun}\longrightarrow 1
\end{equation}
\begin{equation}\label{seq2}
1\longrightarrow G_{\alg}\longrightarrow G_{\nun}\longrightarrow G_{N}\longrightarrow 1
\end{equation}
\begin{equation}\label{seq3}
1\longrightarrow G_{L} \longrightarrow G_{\alg}\longrightarrow G_{\ab}\longrightarrow 1
\end{equation}
with the following properties
\begin{itemize}
\item[(i)] the group $G_{\nun}$ is a subgroup
of a group $\Bir(X_{\nun})$ for some quasi-minimal model~$X_{\nun}$
of dimension at most $n$
that depends only on $X$ (but not on the subgroup~$G$);
\item[(ii)] the group $G_{\RC}$ is a subgroup
of a group $\Bir(X_{\RC})$ for some rationally connected
variety $X_{\RC}$ of dimension at most $n$ defined
over the field $\Bbbk(X_{\nun})$;
\item[(iii)] the group $G_{N}$ is a subgroup
of $\Aut(N)$ for the finitely 
generated abelian group~\mbox{$N=\NS^{\W}(X_{\nun})$}
that depends only on $X$;
\item[(iv)] the group $G_{\alg}$
acts (maybe not faithfully) on each of the algebraic
equivalence classes of Weil divisors on $X_{\nun}$;
\item[(v)]
the group $G_{\ab}$ is a
subgroup of a group $\Aut_{g}(A)$
(see Definition~\ref{definition:Aut-g}),
where $A$ is an abelian variety of dimension $q(X_{\nun})= q(X)$;
\item[(vi)] the group $G_{L}$ is a subgroup
of a group $\Bir(X_{\nun})$ that
preserves the class $L$ (cf.~\S\ref{section:minimal-models}).
\end{itemize}
\end{proposition}
\begin{proof}
Let
$$\phi_{\RC}\colon X\dasharrow X_{\nun}$$
be the maximal
rationally connected fibration.
Choose~\mbox{$G_{\RC}\subset G$} to be the maximal subgroup
acting fiberwise with respect to $\phi_{\RC}$.
Let~$X_{\RC}$ be the fiber of $\phi_{\RC}$ over the generic
scheme point of $X_{\nun}$.
Since the maximal rationally connected fibration is
functorial (see~\cite[Theorem~IV.5.5]{Kollar-1996-RC}), the group~$G_{\RC}$ 
acts by birational transformations of~$X_{\RC}$.
Furthermore, the group
$$G_{\nun}=G/G_{\RC}$$
acts by birational
transformations of $X_{\nun}$.
Note that $X_{\RC}$ is a rationally connected variety
over the field~$\Bbbk(X_{\nun})$,
and the variety $X_{\nun}$ is non-uniruled.
We may assume
that~$X_{\nun}$ is a quasi-minimal model (on which
the group~$G_{\nun}$ still acts by birational transformations),
and $L$ is an ample divisor
class on~$X_{\nun}$.
In particular, we have established the exact
sequence~\eqref{seq1} and proved~(i) and~(ii).

Consider the action of $G_{\nun}$ on the group of Weil divisors
$\Cl(X_{\nun})$ and on
the Neron--Severi group $\NS^{\W}(X_{\nun})$ (see~\S\ref{section:minimal-models}).
Let $G_{\alg}\subset G_{\nun}$ be the kernel
of this action. In particular,
the action of $G_{\alg}$ on $\Cl(X_{\nun})$ preserves
each of the algebraic equivalence classes of Weil divisors
on $X_{\nun}$.
Moreover, the group
$$G_{N}=G_{\nun}/G_{\alg}$$
is a subgroup
of the automorphism group of the finitely generated
abelian group~$\NS^{\W}(X_{\nun})$.
Therefore, we have established the exact
sequence~\eqref{seq2} and proved~(iii) and~(iv).

Denote by
$$\Cl_L(X_{\nun})\subset\Cl(X_{\nun})$$
the class of
algebraic equivalence of the divisor~\mbox{$L\in\Cl(X_{\nun})$}.
Recall that $\Cl_L(X_{\nun})$ has a structure of an algebraic variety,
so that $\Cl_L(X_{\nun})$ is a torsor over an abelian variety~$\Cl^0(X)$.
Moreover, the group $G_{\alg}$ acts by automorphisms of the variety
$\Cl_L(X_{\nun})$ (see~\S\ref{section:minimal-models}).
Let $G_{L}\subset G_{\alg}$ be the kernel
of the action of $G_{\alg}$ on $\Cl_L(X_{\nun})$.
In particular, the group $G_{L}$ preserves the class $L$.
Moreover,
the group
$$G_{\ab}=G_{\alg}/G_{L}$$
is a subgroup of~\mbox{$\Aut_{g}\big(\Cl_L(X_{\nun})\big)$}.
Therefore, we have established the exact
sequence~\eqref{seq3} and proved~(v) and~(vi).
\end{proof}

\begin{remark}
Note that in the proof of Proposition~\ref{proposition:technical}
in the case of an algebraically closed field $\Bbbk$
one can avoid using the field $\Bbbk(X_{\nun})$ and choose
$X_{\RC}$ to be a fiber of $\phi_{\RC}$ over a general \emph{closed}
point of $X_{\nun}$. Still in the general case
passing to the field $\Bbbk(X_{\nun})$ looks inevitable
(at least in our approach) since the base $X_{\nun}$ may have
no $\Bbbk$-points at all.
\end{remark}

\begin{remark}
The choice of a birational polarization $L$ on the base of the maximal
rationally connected fibration of~$X$ in
Proposition~\ref{proposition:technical} is auxiliary, and
the main properties of the groups we are going to consider
will not depend on this choice (although the particular
groups arising in the exact sequences~\ref{seq1}, \ref{seq2} and~\ref{seq3}
may depend on $L$).
\end{remark}

\begin{corollary}\label{corollary:technical}
Let $X$ be a variety of dimension $n$.
Let $\G_{\RC}(X)$ and $\G_{\ab}(X)$ be the families of groups
arising in Proposition~\ref{proposition:technical}
as the groups $G_{\RC}$ and $G_{\ab}$, respectively,
for various choices of finite groups $G\subset\Bir(X)$.
Then
\begin{itemize}
\item[(i)]
the group $\Bir(X)$
has bounded finite subgroups provided that $\G_{\RC}(X)$ and $\G_{\ab}(X)$
have uniformly bounded finite subgroups;
\item[(ii)]
the group $\Bir(X)$ is Jordan provided that
$\G_{\ab}(X)$ is uniformly
Jordan and $\G_{\RC}(X)$ has uniformly bounded finite subgroups;
\item[(iii)]
the group $\Bir(X)$ is Jordan provided that
$\G_{\ab}(X)$ has uniformly bounded finite subgroups and $\G_{\RC}(X)$
is uniformly Jordan.
\end{itemize}
\end{corollary}
\begin{proof}
Choose some birational polarization $L$ on the base of the maximal
rationally connected fibration of $X$.
Let $\G_{N}(X)$ and $\G_{L}(X)$ be the families of groups
arising in Proposition~\ref{proposition:technical}
as the groups $G_{N}$ and $G_{L}$, respectively,
for various choices of finite groups~\mbox{$G\subset\Bir(X)$}.
Then the family $\G_{N}$ has uniformly bounded
finite subgroups by Corollary~\ref{corollary:lattice},
and $\G_{L}(X)$ has uniformly bounded
finite subgroups by Lemma~\ref{lemma-minimal-model}.
Therefore, assertion~(i) follows from Proposition~\ref{proposition:technical}
and Lemma~\ref{lemma:BFS-by-BFS}.
Assertion~(ii) follows from
Proposition~\ref{proposition:technical},
Lemmas~\ref{lemma:BR-by-BR}, 
\ref{lemma:Jordan-by-BFS-with-BR} and~\ref{lemma:Jordan-by-BFS}
and Corollary~\ref{corollary:abelian-variety-Jordan}.
Finally, assertion~(iii)
follows from Proposition~\ref{proposition:technical}
and Lemmas~\ref{lemma:BFS-by-BFS} and~\ref{lemma:Jordan-by-BFS}.
\end{proof}

\begin{remark}[{cf.~\cite[Theorem~1.10]{Prokhorov-Shramov}}]
\label{remark:BR}
Arguing as in the proof of Corollary~\ref{corollary:technical},
one can easily show that if $X$ is non-uniruled,
then the group $\Bir(X)$ has finite subgroups of bounded rank.
Moreover, these arguments together with~\cite[Theorem~4.2]{Prokhorov-Shramov}
show that modulo Conjecture~BAB the same assertion holds
for an arbitrary variety~$X$.
\end{remark}

Now we are ready to prove Theorem~\ref{theorem:main}.

\begin{proof}[{Proof of Theorem~\ref{theorem:main}}]
Let $\G_{\RC}(X)$ and $\G_{\ab}(X)$
be the families of groups
defined in Corollary~\ref{corollary:technical}.
By Theorem~\ref{theorem:RC-Jordan} the family
$\G_{\RC}(X)$ is uniformly Jordan.
By Corollary~\ref{corollary:abelian-variety-Jordan}
the family $\G_{\ab}(X)$ is uniformly Jordan.
Moreover, $\G_{\RC}(X)$ consists of trivial groups
(and thus has uniformly bounded finite subgroups)
provided that the variety $X$ is non-uniruled,
and~$\G_{\ab}(X)$ consists of trivial groups
(and thus has uniformly bounded finite subgroups)
provided that $q(X)=0$.
Now the assertions (i), (ii) and (iii)
of Theorem~\ref{theorem:main} are implied by
the assertions (i), (ii) and (iii)
of Corollary~\ref{corollary:technical},
respectively.
\end{proof}

\section{Proof of Theorem~\ref{theorem:Q}}
\label{section:Q}

In this section we use Proposition~\ref{proposition:technical}
to prove Theorem~\ref{theorem:Q}, and derive
Corollary~\ref{corollary:Serre-OK}.

\begin{lemma}\label{lemma:GL-uniform}
Suppose that $\Bbbk$ is a finitely generated field over $\Q$,
and $N$ is a positive integer. Let~$\G$ be the family of groups
$\GL_N(\K)$, where $\K$ varies in the set of finitely generated
fields over~$\Bbbk$ such that $\Bbbk$ is algebraically closed
in $\K$. Then the family $\G$ has uniformly bounded finite subgroups.
\end{lemma}
\begin{proof}
To start with, the family $\G$ is uniformly Jordan. Indeed, any finite
subgroup of $\GL_N(\K)$ is embeddable into, say, $\GL_N(\C)$, so that
the constants appearing in Definition~\ref{definition:uniformly-Jordan}
for the groups $\GL_N(\K)$ are all bounded by the corresponding constant
for~$\GL_N(\C)$.
Therefore, replacing a finite subgroup $G\subset\GL_N(\K)$
by its abelian subgroup of bounded index if necessary, we are left with the
task to bound the order of finite abelian subgroups
of $\GL_N(\K)$.

Suppose that $G\subset\GL_N(\K)$ is a finite abelian subgroup.
Then all elements of $G$ are simultaneously diagonalizable
over the algebraic closure $\bar{\K}$ of the field $\K$.
Thus $G$ is generated by at most $N$ elements.
Therefore, it is enough to show that the orders of the elements
of $G$ are bounded by some constant that does not depend on $\K$ and $G$.

Let $g\in\GL_N(\K)$ be an element of finite order $\ord(g)$.
We claim that $\ord(g)$ is bounded by a constant that depends only on
the field $\Bbbk$ and the integer $N$, but not on the field~$\K$. Indeed,
let $F_g(u)$ be the minimal polynomial of the element $g\in\GL_N(\K)$.
Then~$F_g(u)$ is a polynomial of degree at most $N$
with coefficients from the field $\K$, and~\mbox{$F_g(u)$}
divides~\mbox{$u^{\ord(g)}-1$}. Our goal is to prove 
that the roots of all possible polynomials~$F_g(u)$ 
form a finite set of elements of~$\bar{\Bbbk}$. This will imply that 
the set of possible eigen-values of elements of finite order in~$\GL_N(\K)$
is bounded, which means boundedness of orders of such elements.

Let $\Phi_l(u)$ be the $l$-th cyclotomic polynomial.
Recall that $\Phi_l(u)$ is defined over $\mathbb{Z}$.
Moreover, $\Phi_l(u)$ is either irreducible over $\K$, or is a product of linear
polynomials over $\K$.
As usual, one has
\begin{equation}\label{eq:cyclotomic}
u^{\ord(g)}-1=\prod\limits_{l\, \mid\,  \ord(g)}\Phi_l(u).
\end{equation}

Let $G(u)$ be a non-linear irreducible polynomial over $\Bbbk$
that divides $F_g(u)$. Since~$\Bbbk$ is algebraically closed in $\K$,
we see that $G(u)$ is also irreducible over $\K$.
Note that $G(u)$ coincides with some irreducible polynomial $\Phi_l(u)$.
One has
$$\varphi(l)=\deg\big(\Phi_l(u)\big)=\deg\big(G(u)\big)\le
\deg\big(F_g(u)\big)\le N,$$
where $\varphi(l)$ is the Euler function of $l$.
Therefore, the numbers $l$ appearing for the \emph{irreducible}
polynomials $\Phi_l(u)$ in~\eqref{eq:cyclotomic}
are bounded in terms of $N$. On the other hand,
if~$\Phi_l(u)$ is reducible over $\K$, then $\K$ (and thus also $\Bbbk$)
contains a primitive root of unity
of degree~$l$. Hence the number of polynomials
like this appearing in~\eqref{eq:cyclotomic} is bounded by some constant
(that depends only on $\Bbbk$) because the field $\Bbbk$ is
finitely generated.
Since the polynomial~\mbox{$u^{\ord(g)}-1$}, and thus also 
the polynomial~$F_g(u)$, 
has no multiple roots, we conclude that~$F_g(u)$ 
is a product of a bounded number
of cyclotomic polynomials of bounded degrees.
Therefore, only a finite number of elements of~$\bar{\Bbbk}$
can be roots of $F_g(u)$ for various $g$, 
and the assertion of the lemma follows.
\end{proof}

\begin{remark}
If $X$ is a (geometrically irreducible) variety
over a field $\Bbbk$, then $\Bbbk$
is algebraically closed in $\K=\Bbbk(X)$.
\end{remark}

An immediate consequence of Lemma~\ref{lemma:GL-uniform} is the following.

\begin{corollary}\label{corollary:alg-group-uniform}
Suppose that $\Bbbk$ is a finitely generated field over $\Q$, and
$\Gamma$ is a linear algebraic group.
Let $\G$ be the family of groups
$\Gamma(\K)$, where $\K$ varies in the set of finitely generated
fields over $\Bbbk$ such that $\Bbbk$ is algebraically closed
in $\K$. Then the family $\G$ has uniformly bounded finite subgroups.
\end{corollary}

\begin{remark}\label{remark:Serre}
Lemma~\ref{lemma:GL-uniform} and Corollary~\ref{corollary:alg-group-uniform}
are also implied by~\cite[Theorem~5]{Serre2009}. Indeed, in the notation
of~\cite[\S 4]{Serre2009}
the values of the invariant~$t$ (the invariant~$m$, respectively)
are the same for the fields~$\Bbbk$
and~$\K$ provided that~$\Bbbk$ is algebraically
closed in~$\K$,
so that in the notation of Lemma~\ref{lemma:GL-uniform} and
Corollary~\ref{corollary:alg-group-uniform}
these invariants are bounded in the family~$\G$.
\end{remark}

To prove Theorem~\ref{theorem:Q} we will start with
an assertion that is more or less
its particular case. 
Recall that a $G$-equivariant morphism $\phi\colon Y\to S$ of normal
varieties acted on by a finite group $G$ is a \emph{$G$-Mori fiber space},
if $Y$ has terminal $G\Q$-factorial singularities, $\dim(S)<\dim(Y)$,
the fibers of $\phi$ are connected, the anticanonical divisor $-K_Y$
is $\phi$-ample, and the relative $G$-invariant Picard number
$\rho^G(Y/S)$ equals~$1$.

\begin{lemma}
\label{lemma:RC-over-Q-BFS}
Suppose that $\Bbbk$ is a finitely generated field over $\Q$.
Let $\G_{\RC}^{\Bbbk}(n)$ be the family of groups~\mbox{$\Bir(X)$},
where $X$
varies in the set
of rationally connected varieties of dimension~$n$ over some field $\K$,
and $\K$ itself varies in the set of finitely generated
fields over $\Bbbk$ such that $\Bbbk$ is algebraically closed
in $\K$. Assume that Conjecture~BAB
holds in dimension $n$.
Then the family $\G^{\Bbbk}_{\RC}(n)$ has uniformly bounded
finite subgroups.
\end{lemma}
\begin{proof}
Let $X$ be a rationally connected variety of dimension $n$ over a
field $\K$,
and let~\mbox{$G\subset\Bir(X)$} be a finite group.
By Lemma-Definition~\ref{regularization} there exists a smooth regularization
$\tilde X$ of~$G$.
Note that $\tilde{X}$ is rationally connected since so is $X$.
Run a $G$-Minimal Model Program on~$\tilde{X}$.
This is possible due to an equivariant version
of~\cite[Corollary~1.3.3]{BCHM} and
\cite[Theorem~1]{MiyaokaMori},
since rational connectedness
implies uniruledness.
We obtain a rationally connected
variety $Y$ birational to $\tilde{X}$ (and thus to $X$)
with a faithful (regular) action of
the group $G$ and a structure
$\phi\colon Y\to S$ of a $G$-Mori fiber space.

Suppose that $\dim(S)=0$. Then $Y$ is a Fano variety with terminal
singularities. Using Conjecture~BAB
and arguing as in the proof of~\cite[Lemma~4.6]{Prokhorov-Shramov}
we see that there exists a positive integer $N=N(n)$
that does not depend on the field $\K$ and on the
variety~$Y$ (and thus also on $X$) such that $G\subset\PGL_N(\K)$.
Therefore, in this case the assertion follows
from Theorem~\ref{theorem:GL-number-field}.

Now suppose that $\dim(S)>0$.
Consider an exact sequence of groups
$$
1\longrightarrow G_{f}\longrightarrow G\longrightarrow G_{b}\longrightarrow 1,
$$
where the action of $G_{f}$ is fiberwise with respect to $\phi$ and
$G_b$ is the image of $G$ in $\Aut(S)$.
We have an embedding
$G_{f}\subset\Aut(Y_{\eta})$, where $Y_{\eta}$
is the fiber of $\phi$ over the generic
scheme point $\eta$ of $S$ (cf. the proof
of Proposition~\ref{proposition:technical}).
Note that $S$ is rationally connected since it is
dominated by a rationally connected
variety $Y$.
Moreover, $Y_{\eta}$ is a Fano variety
with at worst terminal singularities
by~\cite[\S 5-1]{KMM},
so that $Y_{\eta}$ is rationally
connected by~\cite[Theorem~1]{Zhang-Qi-2006}.
Note also that $Y_{\eta}$ is defined over the field
$\K_{\eta}=\K(S)$ that is finitely generated over $\K$
(and thus over $\Bbbk$).
Since $\phi$ has connected fibers,
the field $\K$ is algebraically closed in $\K_{\eta}$, so that
$\Bbbk$ is algebraically closed in $\K_{\eta}$ as well.

Let $\G_{f}$ and $\G_{b}$ be the families of groups
arising in the above procedure
as the groups~$G_{f}$ and~$G_{b}$, respectively,
for various choices of a field $\K$, a variety $X$ and a
finite group~\mbox{$G\subset\Bir(X)$}.
Since $\dim(S)<n$ and $\dim(Y_{\eta})<n$, induction in $n$
shows that both~$\G_b$ and~$\G_{f}$ have universally bounded finite
subgroups. Thus the assertion follows by Lemma~\ref{lemma:BFS-by-BFS}.
\end{proof}

Now we are ready to prove Theorem~\ref{theorem:Q}.

\begin{proof}[{Proof of Theorem~\ref{theorem:Q}}]
Let $\G_{\RC}(X)$ and $\G_{\ab}(X)$
be families of groups
defined in Corollary~\ref{corollary:technical}.
By Lemma~\ref{lemma:RC-over-Q-BFS} the family
$\G_{\RC}(X)$ has uniformly bounded finite subgroups.
By Corollary~\ref{corollary:abelian-variety-BFS}
the family $\G_{\ab}(X)$ also has uniformly bounded finite subgroups.
Therefore, the assertion is implied by Corollary~\ref{corollary:technical}(i).
\end{proof}

Finally, we use Theorem~\ref{theorem:Q} to derive
Corollary~\ref{corollary:Serre-OK}.

\begin{proof}[{Proof of Corollary~\ref{corollary:Serre-OK}}]
Let $\Bbbk$ be the algebraic closure of $\mathbb Q$ in $K$ and
let
$$\Aut_{\Bbbk}(K)\subset \Aut(K)$$
be the maximal subgroup of $\Aut(K)$ that acts trivially on $\Bbbk$.
Then
$\Aut_{\Bbbk}(K)$ is normal in~$\Aut(K)$,
and
$$\Aut(\Bbbk)\simeq\Aut(K)/\Aut_{\Bbbk}(K)$$
is a finite group.
Thus it is sufficient to show that $\Aut_{\Bbbk}(K)$
has bounded finite subgroups.
Let $R\subset K$ be a finitely generated $\Bbbk$-subalgebra
that such that the field $K$
is the field of fractions of $R$.
The assertion follows by Theorem~\ref{theorem:Q}
applied to the (affine) variety~\mbox{$X=\operatorname{Spec}(R)$}.
\end{proof}

\section{Solvably Jordan groups}
\label{section-Solvably-Jordan-groups}

\begin{definition}
\label{definition:solvably-Jordan}
Let $\G$ be a family of groups. We say that
$\G$ is \emph{uniformly solvably Jordan}
if there exists a constant $J_S=J_S(\Gamma)$ such that
for any group $\Gamma\in\G$ and any finite subgroup
$G\subset\Gamma$ there
exists a solvable subgroup
$S\subset G$ of index at most $J_S$.
We say that a group $\Gamma$
is \emph{solvably Jordan}
if the family $\{\Gamma\}$ is uniformly solvably Jordan.
\end{definition}

\begin{remark}\label{remark:solvably-Jordan}
If a family $\G$ is uniformly Jordan,
then it is uniformly solvably Jordan.
\end{remark}

D.\,Allcock asked the following.

\begin{question}\label{question:solvable}
Which varieties have solvably Jordan groups of birational
automorphisms?
\end{question}

The purpose of this section is to give an answer to
Question~\ref{question:solvable}.

\begin{lemma}
\label{lemma:solvably-Jordan-by-solvably-Jordan}
Let $\G_1$ and $\G_2$ be uniformly solvably Jordan families of groups.
Let $\G$ be a family of groups $G$ such that there exists
an exact sequence
\begin{equation}\label{eq:solvable-sequence}
1\longrightarrow G_1\longrightarrow G\longrightarrow G_2\longrightarrow 1,
\end{equation}
where $G_1\in\G_1$ and $G_2\in\G_2$.
Then $\G$ is uniformly solvably Jordan.
\end{lemma}
\begin{proof}
It is enough to prove the assertion assuming that $G$
(and thus also $G_1$ and $G_2$) is finite.
Replacing $G_2$ by its solvable subgroup 
of bounded index and replacing $G$ by the preimage of the latter
subgroup, we may assume that the group $G_2$ in~\eqref{eq:solvable-sequence}
is solvable. 

Let $S_1\subset G_1$ be the maximal normal solvable 
subgroup of $G_1$. Then $S_1$ is preserved by automorphisms
of $G_1$, and thus $S_1$ is a normal subgroup 
of the group $G$. Since an extension of a solvable 
group by a solvable group is again solvable, and the index of~$S_1$ in~$G_1$ 
is bounded, we may replace $G_1$ and $G$
by $G_1/S_1$ and $G/S_1$, respectively, and assume that 
the order of~$G_1$ in~\eqref{eq:solvable-sequence} is bounded.

Let $C\subset G$ be the centralizer of the subgroup 
$G_1$. Since the subgroup $G_1\subset G$ is normal, we conclude that 
$C\subset G$ is also normal. 
Thus the group $G/C$ embeds 
into the group $\Aut(G_1)$. Since $|G_1|$ is bounded, we
see that $|\Aut(G_1)|$ is bounded as well. This implies that 
the index $[G:C]$ is bounded. Put $H=C\cap G_1$. Then the group 
$H$ is abelian. On the other hand, 
the group $C/H$ embeds into $G_2$, so that $C/H$ is solvable.
Since an extension of a solvable
group by a solvable group is solvable, we see that $C$ is a solvable subgroup 
of $G$ of bounded index.
\end{proof}

\begin{proposition}\label{proposition:solvably-Jordan}
Let $X$ be a variety of dimension $n$.
Suppose that Conjecture~BAB
holds in dimension $n$. Then the group $\Bir(X)$
is solvably Jordan.
\end{proposition}
\begin{proof}
Choose some birational polarization $L$ on the base of the maximal
rationally connected fibration of $X$.
Let $\G_{\RC}(X)$, $\G_{\ab}(X)$, $\G_{N}(X)$ and $\G_{L}(X)$
be the families of groups
arising in Proposition~\ref{proposition:technical}
as the groups $G_{\RC}$, $G_{\ab}$, $G_{N}$ and $G_{L}$, respectively,
for various choices of finite groups $G\subset\Bir(X)$.
Applying Theorem~\ref{theorem:RC-Jordan},
Corollary~\ref{corollary:abelian-variety-Jordan},
Corollary~\ref{corollary:lattice} and
Lemma~\ref{lemma-minimal-model}
together with Remarks~\ref{remark:BFS-implies-Jordan}
and~\ref{remark:solvably-Jordan}, we see
that the families $\G_{\RC}(X)$, $\G_{\ab}(X)$, $\G_{N}(X)$ and $\G_{L}(X)$
are uniformly solvably Jordan.
Now the assertion is implied by
Proposition~\ref{proposition:technical} and
Lemma~\ref{lemma:solvably-Jordan-by-solvably-Jordan}.
\end{proof}

\begin{remark}\label{remark:solvably-Jordan-without-BAB}
Note that for non-unirational varieties the argument
used in the proof of Proposition~\ref{proposition:solvably-Jordan}
does not rely on Conjecture~BAB. Similarly, in dimension $n\le 3$
Proposition~\ref{proposition:solvably-Jordan} also holds
without any additional assumptions
(cf. Corollary~\ref{corollary:dim-3}).
\end{remark}

\section{Discussion}
\label{section:discussion}

In this section we list several open questions
related to the previous consideration, and mention
some possible approaches to them.

\begin{question}\label{question:degenerate-fibers}
Can one use information on degenerate fibers of
certain fibrations to establish Jordan property for automorphism
(resp., birational automorphism) groups?
\end{question}

We note that a typical case that is not covered by Theorem~\ref{theorem:main}
is a variety $X$ with a structure of fibration $\phi\colon X\to X_{\nun}$
such that $\phi$ has rationally connected fibers and $X_{\nun}$ is
a non-uniruled variety with irregularity $q(X_{\nun})=q(X)>0$.
The situation when $X_{\nun}=A$ is an abelian variety, and $\phi$ is a conic
bundle, is already interesting and not completely
accessible on our current level of understanding the geometry of
such fibrations. For example, from~\cite{Zarhin10} we know that
if $X\simeq A\times\P^1$, then the group
$\Bir(X)$ is not Jordan. On the other hand, even in dimension $3$
we are far from being able to analyze even similar examples.
For example, if $\phi$ is a $\P^1$-bundle over an abelian surface $A$, we
do not know how to deal with the Jordan property, except
for the cases that are somehow reduced to the direct product
(say, we do not know if there is a Jordan example of this kind).
Furthermore, if~$\phi$ is a conic bundle with a non-trivial
discriminant $\Delta\subset A$, it seems reasonable to try
(but it is not yet clear for us how) to use the geometry of $\Delta$ to
estimate the image of~$\Bir(X)$ under the natural map $\Bir(X)\to\Aut(A)$.
We expect that
a good starting point here may be to understand the influence
of ampleness of $\Delta$ on Jordan property of $\Bir(X)$.
It is also possible that similar considerations may
help to find out if the Jordan property holds for groups of automorphisms
of affine varieties (cf.~\cite[Question~2.14]{Popov2011}).

The next thing we want to mention is the following.

\begin{question}\label{question:abundance}
Can one use some canonically defined (rational) maps
to provide a more geometric proof of Theorems~\ref{theorem:main}
and~\ref{theorem:Q}?
\end{question}

A general observation is that if a rational map $\phi\colon X\dasharrow X'$
is equivariant with respect to $\Bir(X)$ (or some subgroup
$G\subset\Bir(X)$),
then one has an exact sequence associated with $\phi$
similar to~\eqref{seq1}. This observation was applied to the maximal
rationally connected fibration in the proof of
Proposition~\ref{proposition:technical} and to a $G$-Mori fibration
in the proof of
Lemma~\ref{lemma:RC-over-Q-BFS}.
On the other hand, it is tempting to use other maps that are
canonically defined and thus equivariant. In particular,
it is possible that some information
may be obtained from analysing the Albanese map
$$alb\colon X\dasharrow\Alb(X)$$
and making use of the fact that the target
space $\Alb(X)$ is an abelian variety. We feel that the corresponding
part of our current approach is
somehow ``dual'' to this.
Furthermore, if $X$ is a non-uniruled variety, one can consider a
pluri-canonical map
$$\phi_{can}\colon X\dasharrow X_{can}$$
and make use of the properties of the fibers of
$\phi_{can}$. On the other hand, this approach may be hard to take
since we do not really know much about the fibers of the Albanese map
and about the image of the pluri-canonical map.

The last question we want to discuss is the following.

\begin{question}\label{question:uniform}
Can one prove ``uniform'' analogs of Theorems~\ref{theorem:Q}
and~\ref{theorem:main} for some natural families of varieties,
i.\,e. show that certain families of groups of birational
automorphisms are uniformly Jordan, or have uniformly
bounded finite subgroups?
\end{question}

Of course, such result is not possible in the most general case.
Indeed, for any $m$ the symmetric group $\mathrm{S}_m$ acts by automorphisms
of some curve $C_m$ defined over $\Q$.
Actually, the only (general enough)
results in this direction we are aware of are
Lemma~\ref{lemma:RC-over-Q-BFS} and Theorem~\ref{theorem:RC-Jordan}.
Another result of similar flavour is
Corollary~\ref{corollary:abelian-variety-BFS-torsion}.
A more reasonable version
of Question~\ref{question:uniform} may be the following.

\begin{question}\label{question:invariants}
Can one bound the constants appearing in Theorems~\ref{theorem:Q}
and~\ref{theorem:main} uniformly for some natural families of varieties
in terms of some invariants of these varieties?
\end{question}

A partial answer to Question~\ref{question:invariants}
that illustrates what we would like to know in some more
wide context is a bound on the order of birational automorphism
group of an \mbox{$n$-dimensional} variety of general type
in terms of its canonical volume
(see \cite{Hacon-McKernan-Xu}).
The most general assertion that we may suggest in this direction is as follows.
Suppose that~$\X$ is a family of $n$-dimensional varieties
such that for any $X\in\X$ one can choose a \emph{very ample}
birational polarization $H_X$ on the base of the maximal
rationally connected fibration~\mbox{$\phi_{\RC}\colon X\dasharrow X_{\nun}$}
so that the volume of $H_X$ is bounded by some constant~\mbox{$D=D(\X)$}.
(Recall from Definition~\ref{definition:polarization}
that this polarization is supposed
to be defined not on $X_{\nun}$ itself but on one of its quasi-minimal
models.)
Then in the assumptions of Theorem~\ref{theorem:main}(ii),(iii)
the family $\BB$ of groups $\Bir(X)$, $X\in\X$,
is uniformly Jordan. This directly follows from
an observation that for a family of (polarized) varieties of bounded degree
all essential characteristics of the varieties involved in the
proof of Proposition~\ref{proposition:technical}
(i.\,e. rank and order of torsion of Neron--Severi group
and irregularity)
are bounded.
Similarly, in the assumptions of
Theorems~\ref{theorem:main}(i) and~\ref{theorem:Q}
the family $\BB$ has uniformly bounded finite subgroups,
provided that Conjecture~\ref{conjecture:torsion} holds in dimension~$d$
that equals the maximal irregularity for
varieties $X\in\X$. This follows by the same argument as above
together with Corollary~\ref{corollary:abelian-variety-BFS-torsion}.

\end{document}